\documentclass[12pt,letterpaper]{amsart}
\usepackage{graphicx}
\usepackage{amssymb}
\usepackage{amsthm}
\usepackage{bbm}
\usepackage[all]{xy}
\usepackage{enumerate}
\usepackage{color}
\usepackage{amsmath}
\usepackage{mathrsfs,txfonts}
\usepackage{microtype} %<--- Melhora aparência do texto sensivelmente

\numberwithin{equation}{section}

\DeclareFontFamily{U}{mathc}{}
\newcommand{\Oh}{\mathscr{O}}

\def\mult{\operatorname{mult}}
\def\min{\operatorname{min}}

\def\deg{\operatorname{deg}}

\def\max{\operatorname{max}}
\def\lim{\operatorname{lim}}

\def\pr{{\mathbbm P}}
\def\k{\mathbbm{k}}
\def\C{\mathbbm{C}}

\def\M{{\mathfrak{m}}}
\def\R{\C\{x_1,\ldots,x_n\}}

\def\mult{{\rm mult}}

\providecommand{\deg}{\mathop{\rm deg}\nolimits}

\newtheorem{theorem}{Theorem}[section]
\newtheorem{lemma}[theorem]{Lemma}
\newtheorem{corollary}[theorem]{Corollary}
\newtheorem{definition}[theorem]{Definition}
\newtheorem{proposition}[theorem]{Proposition}
\newtheorem{example}[theorem]{Example}
\newtheorem{remark}[theorem]{Remark}

\newtheorem{question}[theorem]{Question}

\title{On Tjurina Ideals of Hypersurface Singularities}
\author{Jo\~ao H\' elder Olmedo Rodrigues}

\address{\noindent Instituto de Matem\'atica e Estat\'istica\\
Universidade Federal Fluminense\\ 
\noindent R. Prof. Marcos Waldemar de Freitas Reis, Bloco G\\
\noindent Campus do Gragoat\'a - S\~ao Domingos\\
\noindent 24.210-201 - Niter\'oi, RJ, Brazil}

\email{\noindent joaohelder@id.uff.br}

\bigskip
\keywords{\noindent Analytic Classification, Mather-Yau theorem, Hypersurface Singularities, Tjurina Algebras, Tjurina Ideals}

\bigskip
\subjclass{\noindent 2010 Mathematics Subject Classification: 14B05, 14H20, 14D06, 32S05}

\begin{document}

\thanks{The author was partially supported by FAPERJ, ARC 211.361/2019}
%\date{August, 19, 2022}

\begin{abstract}
The Tjurina ideal of a germ of a holomorphic function $f$ is the ideal of $\mathscr{O}_{\mathbbm{C}^n,0}$ - the ring of those germs at $0\in\C^n$ - generated by $f$ itself and by its partial derivatives. Here it is denoted by $T(f)$. The ideal $T(f)$ gives the structure of closed subscheme of $(\C^n,0)$ to the singular set of the hypersurface defined by $f$, being an object of central interest in Singularity Theory. In this note we introduce \emph{$T$-fullness} and \emph{$T$-dependence}, two easily verifiable properties for arbitrary ideals of germs of holomorphic functions. These two properties allow us to give necessary and sufficient conditions on an ideal $I\subset \Oh_{\C^n,0}$ for the equation $I=T(f)$ to admit a solution $f$. As a result we characterize closed subschemes of $(\C^n,0)$ arising as singularities of germs of hypersurfaces.

\end{abstract}
\maketitle

\section{Introduction}
\label{I}

Motivation for the present work comes from the classification of germs of hypersurface singularities in $(\C^n,0)$. A germ of a complex hypersurface $(X,0)\subset (\C^n,0)$  at the origin $0\in\C^n$ is defined as the zero set of some - non-trivial - principal ideal $I(X)$ of $\Oh_{\C^n,0}$, the ring of germs of holomorphic functions at $0\in\C^n$. Let $\M$ denote the unique maximal ideal of $\Oh_{\C^n,0}$. A generator $f\in \M$ - which is well-defined, modulo multiplication by an invertible element in $\Oh_{\C^n,0}$ - of $I(X)=\langle f\rangle $ is said to be an \emph{equation} for $(X,0)$; we often say in this case that $(X,0)\subset (\C^n,0)$ \emph{is defined by} $f$. 

As it is customary, we want to decide when two objects - germs of hypersurfaces at $0\in\C^n$, in the present case - are \emph{alike}, in some prescribed sense. A quite natural notion is the \emph{biholomorphic} equivalence of germs of hypersurfaces. We say that two germs of hypersurfaces, $(X_1,0)\subset (\C^n,0)$ and $(X_2,0)\subset (\C^n,0)$ with equations $f_1,\,f_2$ are \emph{biholomorphically equivalent} if there exist small open neighbourhoods $U_1$ and $U_2$ of the origin $0\in\C^n$, where $f_1,\,f_2$ converge and a (germ of) biholomorphism $\Phi:U_1\rightarrow U_2$ - taking the origin to itself - such that $\Phi(X_1\cap U_1)=X_2\cap U_2$. This is clearly an equivalence relation among germs of hypersurfaces. We refer to the totality of germs of hypersurfaces all biholomorphically equivalent to one another as a \emph{biholomorphic class}. It is easy to see that $(X_1,0)$ and $(X_2,0)$ as above are biholomorphically equivalent if and only if there exists an automorphism $\varphi$ of $\Oh_{\C^n,0}$ and an invertible element $u\in\Oh_{\C^n,0}$ such that $\varphi(f_1) = uf_2$. It is often said in this case that $f_1$ and $f_2$ are \emph{contact equivalent}. So, a guiding big problem is to look for a complete set of invariants describing completely biholomorphic classes. For instance, one of the most famous invariants of a biholomorphic class is the \emph{multiplicity} of their elements. We recall that if a system $\{x_1,\ldots,x_n\}$ of generators of $\M\subset \Oh_{\C^n,0}$ is chosen and if $(X,0)\subset (\C^n,0)$ is a germ of hypersurface defined by $f$, its multiplicity $\mult(X,0)=\mult(f)$ is the smallest degree $m$ of a non-zero homogeneous polynomial appearing in a series expansion $f(x_1,\ldots,x_n)=f_m(x_1,\ldots,x_n)+f_{m+1}(x_1,\ldots,x_n)+\ldots\,\,$ Clearly the multiplicity of a germ doesn't depend on the choice of $\{x_1,\ldots,x_n\}$. So, two biholomorphically equivalent germs of hypersurfaces in $(\C^n,0)$ have the same multiplicity, but not the other way around.

Given the previous terminology, the starting point for our research is a result of J. Mather and S. Yau's study \cite{MY}, which involves a certain \linebreak $\C$-algebra associated to a germ of hypersurface $(X,0)\subset (\C^n,0)$. Fix a system $\{x_1,\ldots,x_n\}$ of minimal generators of $\M\subset \Oh_{\C^n,0}=\C\{x_1,\ldots,x_n\}$. Suppose $(X,0)$ is defined by $f(x_1,\ldots,x_n)$. We define the \emph{Tjurina Algebra} of $(X,0)$ - or more accurately of $f$ - as the quotient ring $$ A(X) = A(f) = \R/\langle f,\frac{\partial f}{\partial x_1},\ldots, \frac{\partial f}{\partial x_n}\rangle.$$ The ideal $\langle f,\frac{\partial f}{\partial x_1},\ldots, \frac{\partial f}{\partial x_n}\rangle$ appearing as the denominator is called the \emph{Tjurina ideal} - or extended Jacobian ideal - of $f$ and we denote it as $T(f)$. If $g$ is another generator of $I(X)$ then $g=uf$, for some invertible element $u\in\R$. Hence $A(f)=A(g)$, because actually $T(f)=T(g)$ and this shows that $A(X)$ really doesn't depend on the chosen generator for $I(X)$.

In order to explain Mather-Yau's result, we observe that when $(X_1,0)$ and $(X_2,0)$ are biholomorphically equivalent germs of hypersurfaces, then from a relation of type $\varphi(f_1) = uf_2$ as above, it is quite straightforward to check that the Tjurina algebras $A(X_1)$ and $A(X_2)$ are isomorphic as \linebreak $\C$-algebras. The converse holds but it is far more subtle and originally proved for the class of \emph{isolated hypersurface singularities}. The result is generally referred to as the Mather-Yau theorem. There is also a generalization due to Gaffney and Hauser (cf. \cite{GH} for details), which is valid for the class of \emph{hypersurface singularities of isolated type}. This class includes (strictly) isolated hypersurfaces singularities. The aforementioned result is stated here in a simplified manner:

\begin{theorem}[Mather, Yau;\, Gaffney, Hauser]\label{MY}
\emph{Two germs of complex hypersurfaces of isolated type, $(X_1,0)\subset (\C^n,0)$ and $(X_2,0)\subset (\C^n,0)$ are biholomorphically equivalent if and only if $A(X_1)$ and $A(X_2)$ are isomorphic as $\C$-algebras.}
\end{theorem}

That's why Mather and Yau originally baptised $A(X)$ as the \emph{moduli algebra} of the germ of hypersurface $(X,0)\subset (\C^n,0)$: the preceding result tells us that the problem of the classification of germs of hypersurfaces singularities of Isolated Type $(X,0)\subset (\C^n,0)$ up to biholomorphic equivalence is equivalent to that of the classification of their Tjurina algebras $A(X)$ up to $\C$-algebra isomorphism.

In this vein we ask ourselves what seems to be a prior question - algebraic of course, as the methods we will use to attack it -, namely:

\begin{question}\label{pergunta} Fixed $n\geqslant 1$, how to find necessary and sufficient conditions for an ideal $I\,\subset\Oh_{\C^n,0}$ to be the Tjurina ideal of some $f\in\M$?
\end{question}

Clearly, for all $n\geqslant 1$, the zero ideal and the whole ring $\Oh_{\C^n,0}$ are Tjurina ideals of $0$ and $x_1$, say, respectively. 

\begin{example} If $n=1$ then any ideal $I\subset \Oh_{\C^1,0}=\C\{x_1\}$ is of the form $I=\M^k=\langle x_1^k\rangle$, $k\geqslant 0$; this is the Tjurina ideal of $f(x_1)=x_1^{k+1}$.
\end{example}

So, if $n=1$, Question \ref{pergunta} has a trivial answer: any ideal is the Tjurina ideal of some function. However, beginning at $n=2$, this is not true any more. Clearly a necessary condition for an ideal $I\,\subset\Oh_{\C^n,0}$ to be the Tjurina ideal of some $f\in \Oh_{\C^n,0}$ is that the minimal number of generators of $I$ should be at most $n+1$, but this is not sufficient as we see in the next

\begin{example}\label{ex1} Let $I=\M^2=\langle x_1^2,x_1x_2,x_2^2\rangle\subset\C\{x_1,x_2\}=\Oh_{\C^2,0}$ - i.e., the ideal defining a planar triple point - is not of the form $T(f)$, for any \linebreak $f(x_1,x_2)\in\Oh_{\C^2,0}$. Indeed, $\dim_{\C}\,\Big(\frac{\Oh_{\C^2,0}}{I}\Big)=3$ and it is known (cf. \cite{H1}, Example 14.1) that there exists - up to contact equivalence - only one $f(x_1,x_2)$ such that $\dim_{\C}\,\Big(\frac{\Oh_{\C^2,0}}{T(f)}\Big)=3$, namely $x_1^4+x_2^2$. Now in $\frac{\Oh_{\C^2,0}}{I}$ any non-invertible element $\xi$ satisfies $\xi^2=0$ while in $\frac{\Oh_{\C^2,0}}{T(x_1^4+x_2^2)}$, $\eta=\bar{x_1}$, the residual class of $x_1$, is such that $\eta^2\neq 0$.\end{example}

\begin{example}\label{nretas} Let $n\geqslant 2$ and consider $I=\langle x_1x_2,\ldots,x_{n-1}x_n\rangle =\langle\{x_ix_j\}_{i<j}\rangle$, an ideal defining the union of $n$ lines in $(\C^n,0)$ meeting transversally at $0$. If $n\geqslant 4$, then $I$ is not a Tjurina ideal of any $f\in\Oh_{\C^n,0}$. Indeed, $I$ is minimally generated by $\frac{n(n-1)}{2}$ monomials, which exceeds $n+1$ if $n\geqslant 4$. If $n=3$ we see that for $f=x_1x_2x_3$, $T(f)=\langle x_2x_3,x_1x_3,x_1x_2\rangle=I$. We will show below (Proposition \ref{principal4}) that in the case $n=2$, the principal ideal $I=\langle x_1x_2\rangle$ is \emph{not} a Tjurina ideal.
\end{example}

Notice that for a given $f\in \M\subset\Oh_{\C^n,0}$ the ideal $T(f)$ defines the \emph{singular (closed) subscheme} of the germ $(X,0)\subset (\C^n,0)$ defined by $f$. So our Question \ref{pergunta} above has an algebraic-geometric counterpart: which closed subchemes of $(\C^n,0)$ are singular subschemes of germs of hypersurfaces in $(\C^n,0)$? With this geometric viewpoint it is of course important to emphasize that the closed subscheme of $(\C^n,0)$ defined by some ideal $I$ is in general non-reduced and shall not be confused with the one defined by the radical $\sqrt{I}$, which is reduced. For example, in $\Oh_{\C^2,0}$ the ideal $\M$ is the Tjurina ideal of $f=x_1x_2$ but the previous Example \ref{ex1} has shown that the ideal $\M^2$ is not a Tjurina ideal.

With this scheme theoretic phrasing, our question was mentioned before by Hauser in \cite{Ha} and addressed by Aluffi in \cite{A} using what he refers to as the \emph{$\mu$-class} of the closed subscheme, in an appropriate Chow group. Aluffi, in (loc.cit.) obtains obstructions for a closed scheme to be the singular scheme of some hypersurface and out of his considerations, several interesting \emph{non-examples} - subschemes that are \emph{not} singular subschemes of hypersurfaces - are presented. The $\mu$-class is defined only for singular schemes of hypersurfaces. Our approach works for \emph{any} closed subscheme of $(\C^n,0)$ (or rather any ideal of $\Oh_{\C^n,0}$) providing sufficient and necessary conditions for an ideal $I$ to be a Tjurina ideal.

In the literature (see \cite{E},\,\cite{IK},\,\cite{Y}) our question is also related to the long standing \emph{Recognition} and \emph{Reconstruction} problems. The Recognition problem is to decide whether a given commutative $\C$-algebra is isomorphic to the Tjurina algebra of some hypersurface singularity, whereas the Reconstruction problem is to find an equation for the hypersurface singularity out of which some given Tjurina algebra arises. We remark that a corollary of our approach is a solution to the problems of Recognition and Reconstruction. 

We now describe the contents of the paper. In Section \ref{primeira}, along with some preliminaries on Tjurina ideals, we introduce and investigate some properties of the set $\Delta(I)$ of antiderivatives of an arbitrary ideal $I$, which is a natural set of elements of $\Oh_{\C^n,0}$ to look at in the search for a solution $f$ for the equation $I=T(f)$. We also suggest an easily applicable method for computation of $\Delta(I)$ in examples, with routines already implemented in \emph{SINGULAR} \cite{DGPS}. In Section \ref{segunda} we present, by means of the introduction of two simple and quite computable operations on arbitrary ideals of $\Oh_{\C^n,0}$, our main result, namely a characterization of the ideals $I$ for which the equation $I=T(f)$ admits a solution $f$. In other words, we characterize Tjurina ideals giving an explicit answer for Question \ref{pergunta}. In Section \ref{terceira}, some properties of the formation of Tjurina ideals are discussed in the monomial case, in connection with recent results of Epure and Schulze \cite{ES}.

 %%%%%%%%%%%%%%%%%%%%%%%%%%%%%%%%%%%%%%%%%%%%%%%%%%%%%%%%%%%%%%%%%%%%%%%%%%%%%%%%%%%%

                             %section

 %%%%%%%%%%%%%%%%%%%%%%%%%%%%%%%%%%%%%%%%%%%%%%%%%%%%%%%%%%%%%%%%%%%%%%%%%%%%%%%%%%%%

\section{Generalities}\label{primeira}

\subsection{Tjurina Ideals}

As we have indicated in the Introduction, we will be concerned with Tjurina ideals of germs of functions $f\in \Oh_{\C^n,0}$. In what follows we fix, once and for all, a basis $\{x_1,\ldots,x_n\}$ - minimal set of generators  - of the maximal ideal $\M\subset\Oh_{\C^n,0}$; then we identify $\Oh_{\C^n,0}=\R$ and compute $T(f)=\langle f,\frac{\partial f}{\partial x_1},\ldots, \frac{\partial f}{\partial x_n}\rangle$ with respect to this basis. Occasionally, in examples with small $n$ we use $x,\,\,\{x,y\},\,\,\{x,y,z\},$ etc. instead of $x_1,\,\,\{x_1,x_2\},\,\,\{x_1,x_2,x_3\},$ etc. Below we collect some elementary, straightforward properties of the formation of Tjurina ideals:

\begin{remark}\label{easy1} Let $f,\,g\in \Oh_{\C^n,0}$. Then
\item[(i)] $T(f+g)\subseteq T(f)+T(g)$;
\item[(ii)] $T(fg)\subseteq fT(g)+gT(f)\subseteq T(f)T(g)\subseteq T(f)\cap T(g)$.
\item[(iii)] If $\langle f \rangle=\langle g\rangle$ then $T(f)=T(g)$;
%\item[(iv)] DEIXAR ISSO? If $T(f)=T(g)$ and $( f )\subseteq ( g )$, then $(f) =(g)$;
\end{remark}

%\pf The verifications of $(i)$ to $(iii)$ are straightforward. For $(iv)$: since $f=ag$ for some $a\in\Oh_{\C^n,0}$ we have $\mult(f)=\mult(a)+\mult(g)$. But using $T(f)=T(g)$ one deduces from Theorem \ref{MY} that $f$ and $g$ define analytically equivalent germs. This is only possible if $\mult(a)=0$ or, equivalently, if $a$ is an invertible element of $\Oh_{\C^n,0}$. Hence $( f ) =( g )$.\cqd

\noindent Notice that $(iii)$ in the previous Remark allows us to talk about $T$ \emph{of a principal ideal}. It will be useful to be able to extend $T$ to an arbitrary ideal, as follows.
%Notice that the  conclusion of $(iv)$ above is false without the hypothesis $( f)\subseteq ( g)$, because of examples like $f(x,y)=x$, $g(x,y)=y$ or $f(x,y)=x^2+y^2$, $g(x,y)=xy$ in $\Oh_{\C^2,0}=\C\{x,y\}$. Furthermore, 

\begin{definition}\label{defTgeral}Let $J=\langle g_1,\ldots,g_q\rangle$ be an ideal of $\Oh_{\C^n,0}$. We denote by $T(J)$ the ideal $T(g_1)+\ldots+T(g_q)$.
\end{definition}

It is an easy matter, based on general properties (cf. Remark \ref{easy1}), to check that indeed the ideal $T(J)$ doesn't depend on the particular choice of generators for $J$. It shall be referred to as \emph{the Tjurina ideal of $J$}. The properties below are easy to verify.

\begin{remark}
\item[(i)] $J\subseteq T(J)$;
\item[(ii)] If $J_1\subseteq J_2$ is an inclusion of ideals then $T(J_1)\subseteq T(J_2)$;
\item[(iii)] If $\{J_{\lambda}\}_{\lambda\in\Lambda}$ is any family of ideals of $\Oh_{\C^n,0}$ then $T(\sum_{\lambda}J_{\lambda})=\sum_{\lambda}T(J_{\lambda})$;
\item[(iv)] $T(J_1\cap J_2)\subseteq T(J_1)\cap T(J_2)$.
\end{remark}

\medskip

In addition to the preceding considerations relative to Tjurina ideals and in view of our guiding Question \ref{pergunta} in which an ideal $I$ is fixed, it is natural to introduce the concept below:

\begin{definition}\label{defdelta}For any given ideal $I\subseteq \Oh_{\C^n,0}$ we define the \emph{set of antiderivatives} of $I$ by means of \[\Delta(I)=\{f\in \Oh_{\C^n,0}\,|\,T(f)\subseteq I\}.\]
\end{definition}\medskip

An expression for $\Delta(I)$, free of coordinates, can be given. We recall that a $\C$-derivation $\delta$ of $\Oh_{\C^n,0}$ is a $\C$-linear map $\delta$ from $\Oh_{\C^n,0}$ to itself that satisfies the Leibniz' rule: for all $f,\,g\in\Oh_{\C^n,0}$, $\delta(fg)=f\delta(g)+g\delta(f).$ The set $Der_{\C}(\Oh_{\C^n,0})$ of $\C$-derivations of $\Oh_{\C^n,0}$ is naturally a $\Oh_{\C^n,0}$-module, which is freely generated by the partial derivatives $\frac{\partial}{\partial x_i}$, $i=1,\ldots, n$. It is then easy to check that $\Delta(I)$ coincides with \[\{f\in I\,|\,\delta(f)\in I,\,\mbox{for every}\,\,\delta \in\, Der_{\C}(\Oh_{\C^n,0})\,\}.\] This shows in particular that $\Delta(I)$ does not depend on our choice of parameters $\{x_1,\ldots,x_n\}$ used to compute Tjurina ideals.

If $I$ is a non-zero ideal of $\Oh_{\C^n,0}$, we use $\mult(I)$ to denote \linebreak $\min\{\mult(g)\,|\,g\in I\}=\max\{k\,|\,I\subseteq \M^k\}$. Below we summarize some easy properties of $\Delta$. 

\begin{remark}\label{easy}
%\item[(i)] $\Delta(I)=\{f\in I\,|\,\delta(f)\in I,\,\mbox{for every}\,\,\delta \in\, Der_{\C}(\Oh_{\C^n,0})\,\}$;
\item[(i)] $\Delta(I)$ is an ideal of $\Oh_{\C^n,0}$;
\item[(ii)] $I^2\subseteq \Delta(I)\subseteq I$;
\item[(iii)] If $\{I_{\lambda}\}_{\lambda\in\Lambda}$ is any family of ideals of $\Oh_{\C^n,0}$ then $\Delta(\bigcap_{\lambda}I_{\lambda})=\bigcap_{\lambda}\Delta(I_{\lambda})$;
\item[(iv)] If $I\subseteq J$ is an inclusion of ideals in $\Oh_{\C^n,0}$ then $\Delta(I)\subseteq \Delta(J)$;
\item[(v)] If $I\neq \langle 0\rangle$ and $I\neq \langle 1\rangle$ then $\mult(\Delta(I))\geqslant\mult(I)+1$.
\end{remark}

\begin{proof}
The verification of $(i)$ to $(iv)$ is straightforward. When it comes to $(v)$, we always have $\mult(\Delta(I))\geqslant\mult(I)$. If equality holds, take $g\in\Delta(I)$ such that $\mult(g)=\mult(\Delta(I))=\mult(I)$. In this case all partial derivatives $\frac{\partial g}{\partial x_i}$ of $g$ would belong to $I$ and we would obtain $\mult(\frac{\partial g}{\partial x_i})\geqslant\mult(I)=\mult(g)$. This is only possible if $g$ is an invertible element of $\Oh_{\C^n,0}$. Hence $I=\langle 1\rangle$. 
\end{proof} 

\begin{example}\label{potencia} Let $\M^k$, with $k\geqslant 1$, denote the $k$-th power of the maximal ideal of $\Oh_{\C^n,0}$. Clearly $\Delta(\M^k)\supset\M^{k+1}$ and according to Remark \ref{easy}, (v), if $f\in\Delta(\M^k)$ then $\mult(f)\geqslant \mult(\M^k)+1=k+1$. Hence, $f\in\M^{k+1}$. It follows that $\Delta(\M^k)=\M^{k+1}$.\end{example}

\begin{example}\label{principal1}Let $I=\langle f^k \rangle$ be a principal ideal, $k\geqslant 1$, with $f\in\Oh_{\C^n,0}$ an irreducible element. Then $\Delta(I)=\langle f^{k+1}\rangle$: because if $g\in \Delta(I)$, we have $g=af^k$, for some $a\in\Oh_{\C^n,0}$ and then, for all $i=1,\ldots,n$, $\frac{\partial g}{\partial x_i}=f^k.\frac{\partial a}{\partial x_i}+akf^{k-1}\frac{\partial f}{\partial x_i}$ is also a multiple of $f^k$. Hence $f$ must divide $ka\frac{\partial f}{\partial x_i}$, for all $i=1,\ldots,n$. Now, if $f$ divides $\frac{\partial f}{\partial x_i}$ for all $i=1,\ldots,n$ then $\langle \frac{\partial f}{\partial x_1},\ldots, \frac{\partial f}{\partial x_n}\rangle\subseteq \langle f \rangle$ and we would deduce as in \cite{B}, Lemma 1.2.13, that $f=0\in\Oh_{\C^n,0}$, which is not the case. So, for some $i$, $f$ does not divide $\frac{\partial f}{\partial x_i}$ and, being an irreducible element, $f$ must divide $a$. Hence $g\in\langle f^{k+1}\rangle$.\end{example}

\begin{remark} Notice that if one is interested in fields $\k$ of positive characteristic $p$, the preceding two Examples are false, being enough to take \linebreak $I=\langle x_1^p\rangle\subset \k[[x_1]]$ as counterexamples. In this case $\Delta(I)=I\neq \langle x_1^{p+1}\rangle$.
\end{remark}

\begin{proposition}\label{primarios} If $\mathfrak{p}\subset\Oh_{\C^n,0}$ is a prime ideal and $I\subset\Oh_{\C^n,0}$ is a $\mathfrak{p}$-primary ideal, then its ideal of antiderivatives $\Delta(I)$ is also a $\mathfrak{p}$-primary ideal.\end{proposition}

\begin{proof}
We must check that $\Delta(I)$ is primary, that is, we will show that if $f,g\in \Oh_{\C^n,0}$, $fg\in\Delta(I)$ and $g\not\in\Delta(I)$, then $f\in\sqrt{\Delta(I)}$. We argue as follows: if $g\not\in\Delta(I)$ there are two possibilities to consider. The first one is $g\not\in I$. In this case, as $I$ is primary, and $fg\in I$ there exists some $k\geqslant 1$ such that $f^k\in I$. The second possibility is that $g\in I$ and then it is forced that, for some $i=1,\ldots, n$, $\frac{\partial g}{\partial x_i}\not\in I$. From $\frac{\partial (fg)}{\partial x_i}\in I$ we deduce $f\frac{\partial g}{\partial x_i}\in I$. Again using that $I$ is primary, there exists some $k\geqslant 1$ such that $f^{k}\in I$. So we see that in any of the two possibilities, $f^k\in I$. It is then easy to check that $f^{k+1}\in\Delta(I)$ as we wanted. Now, the assertion on being $\mathfrak{p}$-primary is a consequence of Remark \ref{easy}, (ii). 
\end{proof}

By taking primary decomposition of ideals, Remark \ref{easy},(iv) and the previous result imply that when it comes to compute $\Delta(I)$ it is possible to restrict attention to primary ideals. One instance where this observation can be applied directly is presented in the example below.

\begin{example}\label{principal2} Let $I=\langle f\rangle $ be a non-trivial principal ideal and $f=f_1^{k_1}\ldots f_r^{k_r}$ be a factorization of $f$ into irreducible, non-associated, elements with positive $k_1,\ldots,k_r$. Then $I=\langle  f_1^{k_1}\rangle \cap\ldots\cap \langle  f_r^{k_r}\rangle $. It follows from Remark \ref{easy}, (iii) and Example \ref{principal1} that $\Delta(I)=\langle f_1^{k_1+1}\rangle \cap\ldots\cap \langle  f_r^{k_r+1}\rangle =\langle  f_1^{k_1+1}\ldots f_r^{k_r+1}\rangle $.
\end{example}

\subsection{Computation of $\Delta(I)$}

To this point we have computed explicitly the ideal of antiderivatives $\Delta(I)$ only in the case where $I$ is a power of the maximal ideal (cf. Example \ref{potencia}) or $I$ is a principal ideal (cf. Example \ref{principal2}).  Now we suggest a method to compute $\Delta(I)$ for an arbitrary ideal $I\subset \Oh_{\C^n,0}$. Our procedure will be illustrated with a number of examples obtained using basic routines already implemented in the software \emph{SINGULAR}, \cite{DGPS}.

Fix a basis $\{x_1,\ldots,x_n\}$ - minimal set of generators  - of the maximal ideal $\M\subset\Oh_{\C^n,0}$; we identify $\Oh_{\C^n,0}=\R$ as before. A free $\Oh_{\C^n,0}$-module of certain rank $\ell$ will be denoted by $F_{\ell}$; also, for any $\underline{b}=(b_1,\ldots,b_{\ell})^{t}$, $\underline{c}=(c_1,\ldots,c_{\ell})^t\in F_{\ell}$ we put $\underline{b}\cdot\underline{c}$ to denote the element $\sum_{k=1}^{\ell}b_k c_k\in \Oh_{\C^n,0}$.

We assume the ideal $I$ given by generators $I=\langle f_1,\ldots,f_s\rangle $. Then some element $g=\underline{a}\cdot\underline{f} = \sum_k a_kf_k\in I$ is an element of $\Delta(I)$ if and only if, for all $j=1,\ldots,n$, the partial derivative $$\frac{\partial g}{\partial x_j}=\sum_k a_k \frac{\partial f_k}{\partial x_j}+\sum_k \frac{\partial a_k}{\partial x_j}f_k=\underline{a}\cdot\frac{\partial \underline{f}}{\partial x_j}+\frac{\partial \underline{a}}{\partial x_j}\cdot\underline{f}$$ belongs to $I$. In other words, $g=\underline{a}\cdot\underline{f}\in \Delta(I)$ if and only if $\underline{a}\cdot\frac{\partial \underline{f}}{\partial x_j}\in I,$ for all $j=1,\ldots,n$.

For each $j=1,\ldots,n$ we denote by $M^j$ the submodule of $F_{s}$ consisting of the elements $\underline{a}\in F_{s}$ such that $\underline{a}\cdot\frac{\partial \underline{f}}{\partial x_j}\in I$. Then $M:=\bigcap_{j=1}^n\,M^j\subseteq F_s$ is a finitely generated $\Oh_{\C^n,0}$-submodule of $F_s$ consisting of elements $\underline{a}\in\,F_{s}$ such that $g=\underline{a}\cdot\underline{f}\in\Delta(I)$. Let $\underline{m}_{\,1},\ldots,\underline{m}_{\,q}\in F_s$ be generators of $M$: $\underline{m}_{\,i}=(m_{1i},\ldots,m_{si})^t$, where $m_{ki}\in \Oh_{\C^n,0}$.\smallskip

Notation being as above we have

\begin{proposition}\label{compdelta} The ideal of antiderivatives $\Delta(I)$ is generated by $\underline{m}_{\,i}\cdot\underline{f}$, for $i=1,\ldots,q$.
\end{proposition}

\begin{proof}

As discussed above, $g\in \Delta(I)$ if and only if $g=\underline{a}\cdot\underline{f}\in \Delta(I)$ for some $\underline{a}\in F_s$ such that $\underline{a}\cdot\frac{\partial \underline{f}}{\partial x_j}\in I,$ for all $j=1,\ldots,n$. That is, if and only if there exists some $\underline{a}\in M$ such that $g=\underline{a}\cdot\underline{f}$. Since $M$ is generated by the $\underline{m}_{\,i}$, the result follows from the $\Oh_{\C^n,0}$-linearity of the $(-)\cdot\underline{f}$ product.
\end{proof} 

Here we illustrate how we used the software \emph{SINGULAR}, \cite{DGPS} to compute the ideal of antiderivatives of a given ideal. We do not claim originality, since we only applied routines already implemented by the software developers and collaborators.

\begin{example}Let $I=\langle x^2,xy,yz,z^2,y^2-xz\rangle\subset \Oh_{\C^3,0}=\C\{x,y,z\}$. With the interface of SINGULAR open, type

\begin{flushleft}
%\texttt{LIB "hnoether.lib";}\\ 
\texttt{ring r=0,(x,y,z),ds;}
\end{flushleft}

\noindent This declares you are working over a field of characteristic zero (the field of rational numbers, indeed), variables $x,y,z$ and says we work over the corresponding ring of power series. We now declare the generators of the ideal $I$ by means of a matrix with one row and (in the present case) five columns: type \begin{flushleft}
\texttt{matrix B[1][5]=x2,xy,yz,z2,y2-xz;}
\end{flushleft}

\noindent Now to compute the submodule $M^1$, (same notation as above), we declare a matrix with the partial derivatives of the given generators in terms of $x$ as entries. Then compute $M^1$ as follows: 

\begin{flushleft}
 \texttt{matrix A1[1][5]=diff(B,x); def m1=modulo(A1,B)};
 \end{flushleft}
 
\noindent Likewise, compute $M^2$ and $M^3$ by taking partial derivatives with respect to $y$ and $z$:

\begin{flushleft}
\texttt{matrix A2[1][5]=diff(B,y); def m2=modulo(A2,B);}\\
\texttt{matrix A3[1][5]=diff(B,z); def m3=modulo(A3,B);}
\end{flushleft}

\noindent Now, we put

\begin{flushleft}
\texttt{def m=intersect(m1,m2,m3);}\\
\texttt{def M=std(m);}
\end{flushleft}

\noindent This is the definition of $M$ and the computation of a standard basis, with respect to the given monomial order. Now, to obtain $M$, type
\begin{flushleft}
\texttt{print(M);}
\end{flushleft}

\noindent In this case \emph{SINGULAR}, \cite{DGPS} gives

\medskip

\noindent\makebox[\textwidth]{%
\small
\(
\left( \begin{array}{ccccccccccccccc}
z & 0 & 0 & -y & x & 0 & 0 & 0 & -z & y & 0 & 0 & 0 & 0 & z^2\\
0 & 0 & -z & 2x & 0 & -2z & 0 & 0 & 2y & 0 & 0 & 0 & 0 & z^2 & 0\\
0 & -2y & 3x & 0 & 0 & 0 & 0 & y & 0 & 0 & 0 & z & 0 & 0 & 0\\
0 & x & 0 & 0 & 0 & 0 & y & 0 & 0 & 0 & z & 0 & 0 & 0 & 0\\
2x & 0 & -y & 0 & 0 & y & 0 & z & 0 & 0 & 0 & 0 & z^2 & 0 & 0\\
\end{array} \right)\)}

\medskip

\noindent Hence, as in the previous Proposition, we obtain generators of $\Delta(I)$ computing $$(-)\cdot(x^2,xy,yz,z^2,y^2-xz),$$ for all (transposed) columns of the previous matrix. To speed up the calculations we can type 

\begin{flushleft}
%\texttt{print(B);}\\
\texttt{ideal delta=B*M;}\\
\texttt{ideal Delta=std(delta);}\\
\texttt{Delta;}
\end{flushleft}

The output lists the generators of $\Delta(I)$. In the present case we obtain $$\Delta(I)=\langle x^3,x^2y,2xy^2-x^2z,\,y^3-3xyz,\,2y^2z-xz^2,\,yz^2,z^3,x^2z^2\rangle.$$
\end{example}

\section{Two key conditions}\label{segunda}

\subsection{T-fullness} We now introduce $T$-fullness, a natural concept which together with $T$-dependence (cf. Definition \ref{tdependence} below) is present in our characterization of Tjurina ideals of $\Oh_{\C^n,0}$ (see Theorem \ref{main}). 

\smallskip

\noindent Recall the definition of the Tjurina ideal of an arbitrary ideal of $\Oh_{\C^n,0}$ (cf. Definition  \ref{defTgeral}).

\begin{definition}\label{tfullness}Let $I$ be an ideal of $\Oh_{\C^n,0}$. In general we have $T(\Delta(I))\subseteq I$. We say that \emph{$I$ is $T$-full} if $I=T(\Delta(I))$.\end{definition}

\begin{example}\label{naotfull} Let $I=\langle xy,x^4+y^3\rangle \subset\Oh_{\C^2,0}=\C\{x,y\}$. Following the routine of SINGULAR, \cite{DGPS}, described as in the example used to illustrate Proposition \ref{compdelta},  we compute $$M=\begin{pmatrix}
    4y^2  & 3x^3   & xy   &  y^3  & x^4 \\
    x  & y   & 0 &  0 &  0\\
\end{pmatrix}\subset F_2.$$

\noindent Hence, $\Delta(I)=\langle 5xy^3+x^5,4x^4y+y^4,x^2y^2,xy^4,x^5y\rangle $. Then $xy\not\in T(\Delta(I))=\langle x^2y,xy^2,x^4+y^3,x^5\rangle$, so $I$ is not $T$-full.\end{example}

$T$-fullness is a property that characterizes Tjurina ideals in the special case of principal ideals. This is discussed in detail in the result below:

\begin{proposition}\label{principal4} Let $n\geqslant 2$ and $I\subset \Oh_{\C^n,0}$ be a (non-trivial) principal ideal.
\begin{itemize}
\item[(i)] If $I=\langle f^k \rangle $ with $f\in \Oh_{\C^n,0}$ irreducible, $k\geqslant 1$. Then $I$ is a Tjurina ideal if and only if $f$ has multiplicity one;
\item[(ii)] Let $I=\langle f\rangle $, $f=f_1^{k_1}\ldots f_r^{k_r}$, with non associated irreducible factors $f_i$ and $r\geqslant 2$. Then $I$ is not a Tjurina ideal;
\item[(iii)] $I$ is a Tjurina ideal if and only if $I$ is $T$-full.
\end{itemize} 
\end{proposition}

\begin{proof}
If $f$ has multiplicity one, clearly $I=\langle f^k \rangle =T(f^{k+1})$. Conversely, if $I=\langle f^k \rangle $ with $f$ irreducible then as we have seen in Example \ref{principal1}, $\Delta(I)$ consists of multiples of $f^{k+1}$. Hence $T(\Delta(I))=T(f^{k+1})=f^kT(f)=IT(f)$. Therefore, if $I$ is a Tjurina ideal it equals $T(f^{k+1})$ and in this case, by Nakayama's Lemma, we conclude that $T(f)=\langle 1 \rangle =\Oh_{\C^n,0}$. Since $f\in\M$ this last condition can occur only if $f$ has multiplicity one. This is part $(i)$.

Now, let $I=\langle f\rangle $ be a principal ideal and $f=f_1^{k_1}\ldots f_r^{k_r}$ as in $(ii)$. We have seen in Example \ref{principal2} that $\Delta(I)$ consists of multiples of $f_1^{k_1+1}\ldots f_r^{k_r+1}$. Hence $T(\Delta(I))=T(f_1^{k_1+1}\ldots f_r^{k_r+1})$. Therefore, if $I$ is a Tjurina ideal it equals $T(f_1^{k_1+1}\ldots f_r^{k_r+1})$. A routine computation using Leibniz' rule shows that for all $j=1,\ldots,n$, $$\frac{\partial (f_1^{k_1+1}\ldots  f_r^{k_r+1})}{\partial x_j}\in\langle  f_1^{k_1}\ldots f_r^{k_r}\rangle \cdot\langle  \{f_1\ldots f_{i-1}\frac{\partial f_i}{\partial x_j}f_{i+1}\ldots f_r\}_{i=1}^r\rangle .$$ As a result $T(f_1^{k_1+1}\ldots f_r^{k_r+1})\subseteq IK$, where $K$ is the ideal generated by $f_1\ldots f_r$ and $$\{f_1\ldots f_{i-1}\frac{\partial f_i}{\partial x_j}f_{i+1}\ldots f_r\}_{i,j},$$ where $i\in\{1,\ldots,r\}$ and $j\in\{1,\ldots,n\}$. Hence $T(\Delta(I))\subseteq\M^{r-1}I$ and since $r\geqslant 2$ we conclude, using Nakayama's Lemma, that $I$ is not a Tjurina ideal. This finishes part $(ii)$.

Part $(iii)$ is a consequence of the proofs presented for $(i)$ and $(ii)$.
\end{proof}

\begin{example}The previous Proposition \ref{principal4}, $(i)$ shows that there is no $f$ with $T(f)$ equal to, say, $I=\langle x_1^2+x_2^3\rangle $. Besides, $(ii)$ applies for $I=\langle  x_1x_2\rangle $, (the case $n=2$ of Example \ref{nretas}) showing that $I$ is not a Tjurina ideal.\end{example}

The general significance (beyond principal ideals) of $T$-fullness in our investigation is apparent in the next result.

\begin{proposition}\label{tfullehnecess} Let $I$ be an ideal of $\Oh_{\C^n,0}$. If $I$ is a Tjurina ideal then $I$ is $T$-full.\end{proposition}

\begin{proof}
Let $\Delta(I)=\langle g_1,\ldots,g_q\rangle $. If $I$ is a Tjurina ideal, then there exists $f\in\Delta(I)$ such that $I=T(f)$. We may write $f=\sum_{k=1}^q r_k g_k$, for some $r_k\in \Oh_{\C^n,0}$. Using Remark \ref{easy1}, we obtain $I=T(f)\subseteq \sum_{k=1}^q T(r_k g_k)\subseteq \sum_{k=1}^q T(g_k)=T(\Delta(I))\subseteq I$ and equality holds throughout. We conclude that $I$ is $T$-full.
\end{proof}

In the general case $T$-fullness is not sufficient for $I$ to be a Tjurina ideal as shows the next

\begin{example} Let $I=\langle x^2,xy,y^2\rangle =\M^2\subset\Oh_{\C^2,0}=\C\{x,y\}$. We have $\Delta(\M^2)=\M^3=\langle x^3,x^2y,xy^2,y^3 \rangle $, so $T(\Delta(I))=T(x^3)+T(x^2 y)+T(x y^2)+T(y^3)=\M^2=I$ and $I$ is $T$-full. We have seen before that $I$ is not a Tjurina ideal.\end{example}

\subsection{T-dependence} Here we explain the last ingredient which is, together with $T$-fullness (cf. Definition \ref{tfullness}), present in our characterization of Tjurina ideals.

We begin with an example. Consider the ideal $I=\M$ and its ideal of antiderivatives $\Delta(I)=\M^2=\langle x^2,xy,y^2 \rangle $ inside $\Oh_{\C^2,0}=\C\{x,y\}$. We notice that a \emph{general} $\C$-linear combination of the given generators of $\Delta(I)$ is a germ $f$ satisfying $I=T(f)$. By contrast, for \emph{special} linear combinations of the generators, the corresponding $T(f)$ is smaller than $I$.

We aim to coin a concept which encompasses the above example and systematizes the intuition, that of a general linear combination $f$ of the given generators of $\Delta(I)$ has the largest possible $T(f)$ and could reveal whether a given ideal $I$ is a Tjurina ideal. To this end, our results and definitions here are expressed geometrically, since it seemed to us more appropriate to explain these ideas. Therefore, we use standard concepts on schemes, consistent with \cite{H2}, Chapter II, to which we refer for unexplained terminology.

Let $J\subset \Oh_{\C^n,0}$ be any ideal. We regard $J$ as an ideal sheaf on $\mbox{Spec}\,\Oh_{\C^n,0}$. Suppose $J$ is given by generators: $J=\langle g_1,\ldots,g_q\rangle \subset \Oh_{\C^n,0}$; then consider the projective $(q-1)$-space over $\Oh_{\C^n,0}$, namely, $\pr^{q-1}=\mbox{Proj}(S)$, where $S=\Oh_{\C^n,0}[\alpha_1,\ldots,\alpha_{q}]=\bigoplus_{d\geqslant 0}\,S_d$ is the ring of polynomials over $\Oh_{\C^n,0}$, with variables $\alpha_i$, graded so that $\deg\alpha_i=1$, for all $i$.

Let $\pi:\pr^{q-1}\rightarrow \mbox{Spec}\,\Oh_{\C^n,0}$ be the natural morphism of $\mbox{Spec}\,\Oh_{\C^n,0}$-schemes and let $\sigma$ denote the global section $\sum_{i=1}^q\,g_i\alpha_i$ of $\pi^*(J)\otimes \Oh_{\pr^{q-1}}(1)$. Then there is a \emph{Tjurina ideal sheaf} of $\sigma$ on $\pr^{q-1}$, namely $\mathscr{T}(\sigma)$, the sheaf associated to the homogeneous ideal $\langle \sigma,\frac{\partial \sigma}{\partial x_1},\ldots,\frac{\partial \sigma}{\partial x_n}\rangle $ of $S$. Clearly $\mathscr{T}(\sigma)$ is a subsheaf of $\pi^*(T(J))$, so we can consider the quotient sheaf $$\mathscr{F}=\frac{\pi^*(T(J))}{\mathscr{T}(\sigma)}$$ on $\pr^{q-1}$. For later use, we recall that since $\mathscr{F}$ is coherent and $\pr^{q-1}$ is a noetherian scheme, the support $\mbox{Supp}\mathscr{F}$ of $\mathscr{F}$ is a closed subscheme of $\pr^{q-1}$ given by the vanishing of $\Big(\mathscr{T}(\sigma):\pi^*(T(J))\Big)$.

\begin{definition}\label{tdependence} We say that an ideal $J\subset\Oh_{\C^n,0}$ is \emph{$T$-dependent} if\linebreak $\pi^{-1}(\M)\not\subset\mbox{Supp}\mathscr{F}$. Equivalently, $J$ is $T$-dependent if $\Big(\mathscr{T}(\sigma):\pi^*(T(J))\Big)\not\subset \M S$.\end{definition}

We check that $T$-dependence for $J$ is a well-defined concept, being independent on the choice of generators for $J$. For, let the ideal $J$ be given by another system of generators, say $J=\langle h_1,\ldots,h_u \rangle $. Let $\pr^{u-1}=\mbox{Proj}(S')$, being $S'=\Oh_{\C^n,0}[\,\beta_1,\ldots,\beta_{u}]$. Carry the above construction with the obvious morphism $\pi':\pr^{u-1}\rightarrow \mbox{Spec}\,\Oh_{\C^n,0}$ and the section $\sigma'=\sum_{j=1}^u\,h_j\beta_j$ instead of $\pi$ and $\sigma$, obtaining a sheaf $\mathscr{F'}$ on $\pr^{u-1}$.

\begin{lemma}Notation being as above, $\pi^{-1}(\M)\not\subset\mbox{Supp}\mathscr{F}$ if and only if $\pi'^{-1}(\M)\not\subset\mbox{Supp}\mathscr{F'}$.\end{lemma}

\begin{proof} We will show that $\pi^{-1}(\M)\not\subset\mbox{Supp}\mathscr{F}$ implies $\pi'^{-1}(\M)\not\subset\mbox{Supp}\mathscr{F'}$. The other implication can be proved analogously. 

Since the $g_i$'s and the $h_j$'s generate the same ideal $J$ of $\Oh_{\C^n,0}$, we can write for all $i$, $g_i=\sum_jr_{ji}h_j$ with some $r_{ji}\in\Oh_{\C^n,0}$ and at least one $r_{ji}\not\in\M$. We construct an homomorphism $\Phi:S'\rightarrow S$ of graded $\Oh_{\C^n,0}$-algebras (preserving degrees) induced by $\beta_j\mapsto \sum_ir_{ji}\alpha_i$. This gives a natural morphism of $\Oh_{\C^n,0}$-schemes, $\varphi: U\rightarrow\pr^{u-1}$, where $U\subset\pr^{q-1}$ is the complement of the indeterminacy locus of $\varphi$, given by $V(\Phi(\beta_1),\ldots,\Phi(\beta_u))\subset\pr^{q-1}$. Observe that $\pi^{-1}(\M)$ is not contained in the indeterminacy locus of $\varphi$ because $\Phi(\beta_j)\not\in\M S$ for at least one $j$. Notice also that the construction of $\varphi$ implies both $\sigma\vert_{U}=\varphi^*(\sigma')$ and $\pi\vert_{U}=\pi'\circ \varphi$. In particular, $\mathscr{T}(\sigma)\vert_{U}=\varphi^*(\mathscr{T}(\sigma'))$ and $\pi^*(T(J))\vert_{U}=\varphi^*\pi'^*(T(J))$. According to the definition of $\mathscr{F}$ and $\mathscr{F'}$, we deduce $\mathscr{F}\vert_{U}=\varphi^{*}\mathscr{F}'$.

Assume $\pi^{-1}(\M)\not\subset\mbox{Supp}\mathscr{F}$. As observed above $\mbox{Supp}\mathscr{F}$ is closed in the irreducible $\pr^{q-1}$. It follows that there exists some $P\in\pi^{-1}(\M)\cap U$, $P\not\in \mbox{Supp}\mathscr{F}$. Since $U\cap \mbox{Supp}\mathscr{F}=\varphi^{-1}(\mbox{Supp}\mathscr{F}')$, we have $\varphi(P)\not\in\mbox{Supp}\mathscr{F'}$ and $\M S'\subset\Phi^{-1}(\M S)\subset\Phi^{-1}(P)=\varphi(P)$. Therefore $\varphi(P)\in \pi'^{-1}(\M)$, showing that $\pi'^{-1}(\M)\not\subset\mbox{Supp}\mathscr{F'}$.\end{proof}

Now that we have defined $T$-fullness and $T$-dependence, we present some examples showing that the two notions are independent of each other.

\begin{example} Let $J=\M^3=\langle x^3,x^2y,xy^2,y^3\rangle \subset \C\{x,y\}=\Oh_{\C^2,0}$. Clearly $J$ is $T$-full. Let us show that $J$ fails to be $T$-dependent.

Notations being as before, $\pi^*(T(J))$ is the sheaf in $\pr^3$ associated to $T(J)S=\langle x^2,xy,y^2\rangle S.$ \,\,\,Let $\sigma=x^3\alpha_1+x^2y\alpha_2+xy^2\alpha_3+y^3\alpha_4$. Then $$x^3\alpha_1+x^2y\alpha_2+xy^2\alpha_3+y^3\alpha_4,\,3x^2\alpha_1+2xy\alpha_2+y^2\alpha_3,\,x^2\alpha_2+2xy\alpha_3+3y^2\alpha_4$$ are generators of $\mathscr{T}(\sigma)$. In order to investigate the support of $\mathscr{F}$ in $\pr^{3}$ we list below the generators of $\Big(\mathscr{T}(\sigma):\pi^*(T(J))\Big)$, obtained with help of \emph{SINGULAR}, \cite{DGPS}. 
\begin{center}
 $3x^2\alpha_1+2xy\alpha_2+y^2\alpha_3$,\\ $x^2\alpha_2+2xy\alpha_3+3y^2\alpha_4$,\\ $2x\alpha_2^2-6x\alpha_1\alpha_3+y\alpha_2\alpha_3-9y\alpha_1\alpha_4$,\\ $x\alpha_2\alpha_3+2y\alpha_3^2-9x\alpha_1\alpha_4-6y\alpha_2\alpha_4$,\\ $2x\alpha_1\alpha_3^2+y\alpha_2\alpha_3^2-6x\alpha_1\alpha_2\alpha_4-4y\alpha_2^2\alpha_4+3y\alpha_1\alpha_3\alpha_4$,\\ $y\alpha_2^2\alpha_3^2-4y\alpha_1\alpha_3^3-6x\alpha_1\alpha_2^2\alpha_4-4y\alpha_2^3\alpha_4+18x\alpha_1^2\alpha_3\alpha_4+15y\alpha_1\alpha_2\alpha_3\alpha_4.$
 \end{center} Clearly $\pi^{-1}(\M)\subseteq\mbox{Supp}\mathscr{F}$. Hence $J$ is not $T$-dependent.\end{example}
 
\begin{example} Let $J=\langle x^4+y^3,xy\rangle \subset \C\{x,y\}=\Oh_{\C^2,0}$. We have seen before (cf. Example \ref{naotfull}) that $J$ fails to be $T$-full. Let us show that $J$ is $T$-dependent.

As above, $\pi^*(T(J))$ is the sheaf associated to $T(J)S=\M S.$ Let \linebreak $\sigma=(x^4+y^3)\alpha_1+xy\alpha_2$. Then $$(x^4+y^3)\alpha_1+xy\alpha_2,\,4x^3\alpha_1+y\alpha_2,\,3y^2\alpha_1+x\alpha_2$$ are generators of $\mathscr{T}(\sigma)$. We investigate the support of $\mathscr{F}$ in $\pr^{1}$. Below we present the generators of $\Big(\mathscr{T}(\sigma):\pi^*(T(J))\Big)$, again obtained with help of \emph{SINGULAR}, \cite{DGPS}. 

\begin{center}
$x\alpha_2+3y^2\alpha_1$,\\
$y\alpha_2+4x^3\alpha_1$,\\
$\alpha_2^2-12x^2y\alpha_1^2$,\\
$(3y^3-4x^4)\alpha_1$,\\
$x^4\alpha_1$,\\
$x^3y^2\alpha_1$,\\
$x^3y\alpha_1^2$,\\
$x^2y^2\alpha_1^2.$
\end{center} We see that the third generator in the above list does not belong to $\M S$. Hence in this case $\Big(\mathscr{T}(\sigma):\pi^*(T(J))\Big)\not\subset \M S$, showing that $J$ is $T$-dependent.\end{example}

\begin{example}\label{principal5} If $J\subseteq \Oh_{\C^n,0}$ is a principal ideal then $J$ is $T$-dependent. Indeed, unraveling the definition in case $J=\langle f\rangle $ we see that $\mbox{Supp}\mathscr{F}\subset\pr^0$ is empty, opposite to $\pi^{-1}(\M)$ which is not.\end{example}

\subsection{Main Result}

We are ready to state and prove our result characterizing Tjurina ideals.

\begin{theorem}\label{main} Let $I\subset\Oh_{\C^n,0}$ be an ideal. Then $I$ is a Tjurina ideal if and only if $I$ is $T$-full and $\Delta(I)$ is $T$-dependent.
\end{theorem}

\begin{proof} Let $\Delta(I)=\langle g_1,\ldots,g_q\rangle $ and let $\mathscr{F}$ denote the sheaf $\pi^*(T(\Delta(I)))/\mathscr{T}(\sigma)$ on $\pr^{q-1}$, as described before. For any $(\lambda_1,\ldots,\lambda_q)\in\C^q\setminus \{0\}$ we will denote by $p_{\lambda}\subseteq \C[\alpha_1,\ldots,\alpha_{q}]$ the homogeneous prime ideal generated by $\{\lambda_k\alpha_{\ell}-\lambda_{\ell}\alpha_k\}_{k,{\ell}}$. If $P\in D_+(\alpha_{\ell})\subset\pr^{q-1}$ one checks easily that \linebreak $\mathscr{T}(\sigma)_P=T\Big(\sigma/\alpha_{\ell}\Big)S_{(P)}$, where parenthetical notation $-\,_{(P)}$ (here and in what follows) indicates the submodule of degree zero elements of the localization of a graded module at $P\in\pr^{q-1}$. Moreover, if $p_{\lambda}S\subseteq P\in D_+(\alpha_{\ell})$ then $\lambda_{\ell}\in\C\setminus 0$ and we can use Remark \ref{easy1} to obtain the following Estimate on ideals in $S_{(P)}$:

\begin{multline}\label{estimate}
T\Big(\sum_k\Big(\lambda_k-\frac{\lambda_{\ell}\alpha_k}{\alpha_{\ell}}\Big)g_k\Big)S_{(P)}\subseteq\\ \subseteq\sum_k T\Big(\Big(\lambda_k-\frac{\lambda_{\ell}\alpha_k}{\alpha_{\ell}}\Big)g_k\Big)S_{(P)}\subseteq \sum_k \Big(\lambda_k-\frac{\lambda_{\ell}\alpha_k}{\alpha_{\ell}}\Big)T(g_k)S_{(P)}\subseteq IP_{(P)}.
\end{multline}

To prove the direct affirmation in the statement, we assume that $I$ is a Tjurina ideal. We have already seen that $I$ is $T$-full (cf. Proposition \ref{tfullehnecess}); in particular $\mathscr{F}=\pi^*(I)/\mathscr{T}(\sigma)$. Now we show that $\Delta(I)$ is $T$-dependent. Since this is clear if $I$ is principal, particularly if $I=\langle 0\rangle$ - cf. Examples \ref{principal2} and \ref{principal5} - we assume $I=T(f)$, non-principal, with $f=\sum_k r_kg_k$ for certain $r_k\in\Oh_{\C^n,0}$.  We write $r_k=\lambda_k+s_k$, for $\lambda_k\in\C$ and $s_k\in\M$, for all $k=1,\ldots,q$. Since $I\neq \langle 0 \rangle$ we can use Remark \ref{easy},\,(v) to check that $f\not\in\M\Delta(I)$. Hence, at least one of the $r_k$'s is an invertible element of $\Oh_{\C^n,0}$. Hence, using Remark \ref{easy1} again we can (and do) assume $(r_1,\ldots,r_q)=(\lambda_1,\ldots,\lambda_q)\in\C^q\setminus \{0\}$. \linebreak We consider $P=p_{\lambda}S+\M S\in\pr^{q-1}$. Then $P\in\pi^{-1}(\M)$, so that \linebreak $\pi^*(I)_{P}=I_{\pi(P)}\otimes S_{(P)}=I_{\M}\otimes S_{(P)}=IS_{(P)}$. On the other hand (say if $P\in D_+(\alpha_{\ell})$), our assumption $I=T(f)$ together with Estimate (\ref{estimate}) above produces
\begin{multline*}IS_{(P)}=T\Big(\sum_k\lambda_k g_k\Big) S_{(P)}\subseteq\\ \subseteq T\Big(\sum_k\Big(\lambda_k-\frac{\lambda_{\ell}\alpha_k}{\alpha_{\ell}}\Big)g_k\Big)S_{(P)}+T\Big(\lambda_{\ell}\sigma/\alpha_{\ell}\Big)S_{(P)}\subseteq IP_{(P)}+\mathscr{T}(\sigma)_P\subseteq IS_{(P)}.
\end{multline*}
It follows that $IP_{(P)}+\mathscr{T}(\sigma)_P=IS_{(P)}$. Then we can use Nakayama's Lemma to obtain $\pi^*(I)_{P}=IS_{(P)}=\mathscr{T}(\sigma)_P$. Hence $P\not\in \mbox{Supp}\mathscr{F}$, which completes the proof that $\Delta(I)$ is $T$-dependent.

To prove the reverse affirmation in the statement $I$ is assumed to be $T$-full and $\Delta(I)$ is assumed to be $T$-dependent. Then $I=T(\Delta(I))$ and there exists some $P\in\pi^{-1}(\M)\setminus \mbox{Supp}\mathscr{F}$. It follows that $$\pi^*(I)_{P}=\pi^*(T(\Delta(I))_{P}=\mathscr{T}(\sigma)_P.$$ Notice that out of the decomposition $S=\C[\alpha_1,\ldots,\alpha_q]+\M S$ we obtain $P=P\cap\C[\alpha_1,\ldots,\alpha_q]+P\cap\M S=P\cap\C[\alpha_1,\ldots,\alpha_q]+\M S$. Therefore, there exists some $(\lambda_1,\ldots,\lambda_q)\in\C^q\setminus \{0\}$, $p_{\lambda}\supseteq P\cap\C[\alpha_1,\ldots,\alpha_q]$, such that the homogeneous prime ideal $p_{\lambda}S+\M S$ also belongs to $\pi^{-1}(\M)\setminus \mbox{Supp}\mathscr{F}$. %Hence, we may assume $P=p_{\lambda}S+\M S$.

With this simplification - assuming $P=p_{\lambda}S+\M S\in D_+(\alpha_{\ell})$, as before - the displayed equality of stalks above reads $IS_{(P)}=T\Big(\lambda_{\ell}\sigma/\alpha_{\ell}\Big)S_{(P)}.$ The task is to prove $I=T(f)$, for $f=\sum_k\lambda_kg_k$. Estimate (\ref{estimate}) here gives 

\begin{multline*}
IS_{(P)}\subseteq T\Big(\sum_k\Big(\lambda_k-\frac{\lambda_{\ell}\alpha_k}{\alpha_{\ell}}\Big)g_k\Big)S_{(P)}+T\Big(\sum_k\lambda_k g_k\Big)S_{(P)}\subseteq \\\subseteq IP_{(P)}+T\Big(\sum_k\lambda_k g_k\Big)S_{(P)}\subseteq IS_{(P)}.
\end{multline*}
We obtain $IP_{(P)}+T(\sum_k\lambda_kg_k)S_{(P)}=IS_{(P)}$ and by Nakayama's Lemma again we deduce $T(\sum_k\lambda_kg_k)S_{(P)}=IS_{(P)}$. 

It remains to verify that $I\subseteq T(\sum_k\lambda_kg_k)$. Clearly it suffices to check that if $h\in I$ and there are some $u\in T(\sum_k\lambda_kg_k)$ and $G_1,G_2\in S$, homogeneous of same degree with $G_2\not\in P$ such that $G_2h=G_1u$, then $h\in T(\sum_k\lambda_kg_k)$. With this reduction, a simple argument involving divisibility and degree in the UFD $S$ then shows $zh=uH$, for some homogeneous element $H\in S$ and some degree zero factor $z$ of $G_2$. In particular, $z\in\Oh_{\C^n,0}\setminus \M$. Again by degree reasons, $H$ has degree zero, i.e., it belongs to $\Oh_{\C^n,0}$. This shows that $h\in T(\sum_k\lambda_kg_k)$. Hence $I\subseteq T(\sum_k\lambda_kg_k)\subseteq I$ and we conclude that $I$ is a Tjurina ideal of $\Oh_{\C^n,0}$.\end{proof}

Notation as in the proof of Theorem \ref{main} we have the Corollary below, the usefulness of which resides in the fact it says that computing the Tjurina ideal of a sufficiently general $\C$-linear combination of generators of $\Delta(I)$ reveals whether $I\subset \Oh_{\C^n,0}$ is a Tjurina ideal. As a consequence it gives an explicit solution to both the Recognition and the Reconstruction problems mentioned in the Introduction.

\begin{corollary}
Let $I\subset \Oh_{\C^n,0}$ be an ideal and let $\Delta(I)=\langle g_1,\ldots,g_q\rangle $. Then $I$ is a Tjurina ideal if and only if $I=T(\sum_k\lambda_kg_k)$ for $p_{\lambda}$ varying in a (non empty) Zariski open subset of $\pr_{\C}^{q-1}$.
\end{corollary}

\begin{proof} The canonical inclusion $\C\rightarrow \Oh_{\C^n,0}$ induces $\C[\alpha_1,\ldots,\alpha_{q}]\rightarrow S$, a graded inclusion of $\C$-algebras; hence a morphism $\rho:\pr^{q-1}\rightarrow \pr_{\C}^{q-1}$, which is proper and dominant. For the non-trivial part of the statement, assume that $I$ is a Tjurina ideal. From the Theorem, we know that $\Delta(I)$ is $T$-dependent, so that $\pi^{-1}(\M)\setminus \mbox{Supp}\mathscr{F}$ is a non empty Zariski open subset of $\pi^{-1}(\M)$. It follows that its image under $\rho$ is a non-empty open Zariski subset $U\subset\pr_{\C}^{q-1}$. Homogeneous prime ideals $p_{\lambda}$ in $U$ correspond bijectively, via $\rho$, to homogeneous prime ideals of type $p_{\lambda}S+\M S\in\pi^{-1}(\M)\setminus \mbox{Supp}\mathscr{F}$. Since the Theorem also guarantees that $I$ is $T$-full, we can proceed as in its proof to show that $I=T(\sum_k\lambda_kg_k)$, for all $p_{\lambda}\in U$.\end{proof}
%ser proprio eh consequncia de que eh separado (por ser afim) e portanto existe NO MAXIMO UM mapa completando o quadrado do teste valuativo. Que de fato EXISTE um é consequencia de que o mapa residual de R induz o puto!

 %%%%%%%%%%%%%%%%%%%%%%%%%%%%%%%%%%%%%%%%%%%%%%%%%%%%%%%%%%%%%%%%%%%%%%%%%%%%%%%%%%%%

                             %section

 %%%%%%%%%%%%%%%%%%%%%%%%%%%%%%%%%%%%%%%%%%%%%%%%%%%%%%%%%%%%%%%%%%%%%%%%%%%%%%%%%%%%

\section{Remarks and Examples}\label{terceira}

Now that we achieved a complete characterization of Tjurina ideals by means of Theorem \ref{main}, we include some observations regarding intersections of Tjurina ideals. We focus on monomial ideals, since it already contains the kind of phenomenon we aim to illustrate.

\begin{example} \label{monomial1} \emph{Irreducible} monomial ideals are necessarily Tjurina ideals. Recall that irreducible monomial ideals are of the form $\langle x_{i_1}^{e_1},\ldots,x_{i_s}^{e_s}\rangle $, for $\{i_1<\ldots <i_s\}\subseteq\{1,\ldots,n\}$ and all $e_{j}>0$. Hence $$\langle x_{i_1}^{e_1},\ldots,x_{i_s}^{e_s}\rangle =T(x_{i_1}^{e_1+1}+\ldots+x_{i_s}^{e_s+1}).$$\end{example}

The following well-known result (see \cite{HH}, Theorem 1.3.1, Corollary 1.3.2, for instance) motivates us to ask questions about monomial ideals which are also Tjurina ideals.

\begin{proposition}\label{monomial0} Any monomial ideal $I\subset \Oh_{\C^n,0}$ can be written uniquely as a finite and irredundant intersection $I=\cap_{k=1}^m\,Q_k$ of \emph{irreducible} monomial ideals.\end{proposition}

Of course we regard the irreducible monomial ideals as the building blocks of more general monomial ideals. Unfortunately, the result of the previous Example does not extend to arbitrary monomial ideals: \linebreak $\langle x^2,xy,y^2\rangle =\langle x^2,y\rangle \cap \langle x,y^2\rangle \subset \Oh_{\C^2,0}=\C\{x,y\}$ is not a Tjurina ideal (cf. Example \ref{ex1}); even for principal monomial ideals such a result does not hold (cf. Proposition \ref{principal4}).

This is \emph{not} to say that reducible monomial ideals are never Tjurina ideals. For instance, $I=\langle yz,xz,xy\rangle =\langle x,y\rangle \cap\langle x,z\rangle \cap \langle y,z \rangle \subset\Oh_{\C^3,0}=\C\{x,y,z\}$ of Example \ref{nretas} is a Tjurina ideal. Indeed, quite recently Epure and Schulze characterized in their main result of \cite{ES}, functions with monomial Tjurina ideals as those which are right equivalent to Thom-Sebastiani polynomials, solving \emph{Problem} $2^{*}$ in \cite{HaS}, which asks for such a characterization.

In our computations of Example \ref{proaluffi}, the next example will be useful.

\begin{example}\label{monomial2} Let $I=\langle x_{i_1}^{e_1},\ldots,x_{i_s}^{e_s}\rangle $ be an irreducible monomial ideal as in Example \ref{monomial1}. From the  results in Section \ref{primeira}, discussion before Proposition \ref{compdelta} we deduce that $a_1x_{i_1}^{e_1}+\ldots+a_sx_{i_s}^{e_s}\in\Delta(I)$ if and only if $a_kx_{i_k}^{e_k-1}\in I$, for all $k=1,\ldots,s$. Here we used that $I$ is monomial. This, in turn, happens if and only if $a_k\in (I:x_{i_k}^{e_k-1})=\langle x_{i_1}^{e_1},\ldots,x_{i_k},\ldots,x_{i_s}^{e_s}\rangle $. As a result we have $\Delta(I)=\langle x_{i_1}^{e_1+1},\ldots,x_{i_s}^{e_s+1}\rangle +I^2$, a monomial ideal.\end{example}

\begin{remark} The previous Example together with Proposition \ref{monomial0} and Remark \ref{easy}, (iii) show that the ideal of antiderivatives of a monomial ideal is also monomial.\end{remark}

It seems worth examining the Example below, in which a monomial and reducible Tjurina ideal is presented, together with some considerations regarding the dynamics of formation of Tjurina ideals. The Example also provides an example of a Tjurina ideal with an embedded component of codimension $2$, so answering negatively a Question left by Aluffi (\cite{A}, Example 3.1).

\begin{example}\label{proaluffi} Let $f=xw^2-yz^2\in\Oh_{\C^4,0}=\C\{x,y,z,w\}$. It has a monomial Tjurina ideal given by $T(f)=\langle yz,z^2,xw,w^2\rangle =I_1\cap I_2\cap I_3\cap I_4$, where $I_1=\langle z,w\rangle $, $I_2=\langle y,z^2,w\rangle $, $I_3=\langle x,z,w^2\rangle $ and $I_4=\langle x,y,z^2,w^2\rangle $. The ideal $T(f)$ has an embedded component of codimension $2$ at $\langle x,y,z^2,w^2\rangle $.

We notice that $I_1\cap I_2\cap I_3\cap I_4$ is a Tjurina ideal but that some ``partial intersections'' are not. For instance $I_1\cap I_4$ or $I_2\cap I_3$ are not Tjurina ideals as one can easily verify computing numbers of generators: both are minimally $6$-generated.

Harder is to check that $I_1\cap I_2\cap I_3=\langle yz,z^2,xw,zw,w^2 \rangle $ is not a Tjurina ideal. We begin computing its ideal of antiderivatives \linebreak $\Delta(I_1\cap I_2\cap I_3)=\Delta(I_1)\cap\Delta(I_2)\cap\Delta(I_3)$. Using the result of the previous Example we get $\Delta(I_1)=\langle z^2,w^2,zw\rangle ,\,\,\,$ $\Delta(I_2)=\langle y^2,z^3,w^2,yz^2,yw,z^2w\rangle ,\,\,\,$ and \linebreak $\Delta(I_3)=\langle x^2,z^2,w^3,xz,xw^2,zw^2\rangle .$ Hence $$\Delta(I_1\cap I_2\cap I_3)=\langle yz^2,z^3,z^2w,xw^2,zw^2,w^3,xyzw\rangle .$$ It is clear now that $I_1\cap I_2\cap I_3$ is $T$-full but we can check with SINGULAR \cite{DGPS} that $\Delta(I_1\cap I_2\cap I_3)$ is not $T$-dependent. Then Theorem $\ref{main}$, guarantees that $I_1\cap I_2\cap I_3$ is not a Tjurina ideal.
\end{example}

\noindent \textbf{Acknowledgements:} The author is grateful to Marcelo Escudeiro Hernandes, whose interest led to the investigation in the present work.

%%%%%%%%%%%%%%%%%%%%%%%%%%%%%%%%%%%%%%%%%%%%%%%%%%%%%%%%%%%%%%%%%%%%%%%%%%%%%%%%%%%

                             %Bibliography

%%%%%%%%%%%%%%%%%%%%%%%%%%%%%%%%%%%%%%%%%%%%%%%%%%%%%%%%%%%%%%%%%%%%%%%%%%%%%%%%%%%

\end{document}